\documentclass[11pt]{article}
\usepackage{amsmath,amssymb,amsthm,graphicx,fixmath}

\usepackage{algpseudocode}
\usepackage{algorithm}
\usepackage{setspace}
\usepackage{ifthen}
\usepackage{tikz}

\def\Re{\mathbb R}

\providecommand{\remove}[1]{}

\theoremstyle{plain}
\newtheorem{theorem}{Theorem}[section]
\newtheorem{lemma}[theorem]{Lemma}

\newtheorem{proposition}[theorem]{Proposition}
\newtheorem{claim}[theorem]{Claim}

\newtheorem{definition}[theorem]{Definition}
\newtheorem{remark}[theorem]{Remark}

\newcommand{\E}{\mathcal{E}}

\newcommand{\F}{\mathcal{F}}

\makeatletter
\def\@fnsymbol#1{\ensuremath{\ifcase#1\or\ast\or\ddagger\or\mathsection\or\mathparagraph\or\|\or\ast\ast\or\dagger\dagger\or\ddagger\ddagger\fi}}
\makeatother

\begin{document}
	
	\title{On Zarankiewicz's Problem for Intersection Hypergraphs of Geometric Objects\footnote{A preliminary version of the paper, which contains a weaker version of the results, was published at the SoCG 2025 conference~\cite{CKS25}. The improvement in this version is making the dependence of all bounds on $t$ sharp.}}

	\author{%
		 Timothy M. Chan\footnote{Siebel School of Computing and Data Science, University of Illinois at Urbana-Champaign, Urbana, IL 61801, USA.
		Supported by NSF Grant CCF-2224271.
		\texttt{tmc@illinois.edu}
	}
			\and 
			Chaya Keller\footnote{
			Department of Computer Science,
			Ariel University, Israel.
			Research partially supported by Grant 1065/20 from the Israel Science Foundation.
			\texttt{chayak@ariel.ac.il}
		}
		\and Shakhar Smorodinsky\footnote{Department of Computer Science, Ben-Gurion University of the NEGEV, Be'er Sheva 84105, Israel. 
			Partially supported by Grant 1065/20 from the Israel Science Foundation. \texttt{shakhar@bgu.ac.il}}
	}

	\maketitle

	\begin{abstract}
		The hypergraph Zarankiewicz's problem, introduced by Erd\H{o}s in 1964, asks for the maximum number of hyperedges in an $r$-partite hypergraph with $n$ vertices in each part that does not contain a copy of $K_{t,t,\ldots,t}$. Erd\H{o}s obtained a near optimal bound of $O(n^{r-1/t^{r-1}})$ for general hypergraphs. In recent years, several works obtained improved bounds under various algebraic assumptions -- e.g., if the hypergraph is semialgebraic.
		
		In this paper we study the problem in a geometric setting -- for $r$-partite intersection hypergraphs of families of geometric objects. Our main results are essentially sharp bounds for families of axis-parallel boxes in $\mathbb{R}^d$ and families of pseudo-discs. For axis-parallel boxes, we obtain the sharp bound $O_{d,r}(tn^{r-1}(\frac{\log n}{\log \log n})^{d-1})$. The best previous bound
		was larger by a factor of about $(\log n)^{d(2^{r-1}-2)}$. For pseudo-discs, we obtain the bound $O_r(tn^{r-1}(\log n)^{r-2})$, which is sharp up to logarithmic factors. As this hypergraph has no algebraic structure, no improvement of Erd\H{o}s' 60-year-old $O(n^{r-1/t^{r-1}})$ bound was known for this setting. Futhermore, even in the special case of discs for which the semialgebraic structure can be used, our result improves the best known result by a factor of $\tilde{\Omega}(n^{\frac{2r-2}{3r-2}})$.
		
		To obtain our results, we use the recently improved results for the graph Zarankiewicz's problem in the corresponding settings, along with a variety of combinatorial and geometric techniques, including shallow cuttings, biclique covers, transversals, and planarity. 
	\end{abstract}

	\section{Introduction}
	
	\subsection{Background}
	
	\noindent \textbf{Zarankiewicz's problem for graphs.} A central research area in extremal combinatorics is \emph{Tur\'{a}n-type questions}, which ask for the maximum number of edges in a graph on $n$ vertices that does not contain a copy of a fixed graph $H$. This research direction was initiated in 1941 by Tur\'{a}n, who showed that the maximum number of edges in a $K_r$-free graph on $n$ vertices is $(1-\frac{1}{r-1}+o(1))\frac{n^2}{2}$. Soon after, Erd\H{o}s, Stone and Simonovits solved the problem for all non-bipartite graphs $H$. They showed that the maximum number is $(1-\frac{1}{\chi(H)-1}+ o(1))\frac{n^2}{2}$, where $\chi(H)$ is the chromatic number of $H$. 
	
	The bipartite case turned out to be significantly harder. In 1951, Zarankiewicz raised the following Tur\'{a}n-type question for complete bipartite graphs: Given $n,t \in \mathbb{N}$, what is the maximum number of edges in a \emph{bipartite} graph $G$ on $n$ vertices that does not contain a copy of the complete bipartite graph $K_{t,t}$? (Note that asking the question for bipartite graphs $G$ like Zarankiewicz did, leads to the same order of magnitude of the size of the extremal graph like in Tur\'{a}n's original question which considers general graphs $G$, as any graph $G$ with $e$ edges contains a bipartite subgraph with at least $e/2$ edges). This question, known as `Zarankiewicz's problem', has become one of the central open problems in extremal graph theory (see~\cite{Sudakov10}). In one of the cornerstone results of extremal graph theory, K{\H o}v\'{a}ri, S{\'o}s and Tur{\'a}n~\cite{KST54} proved an upper bound of $O(n^{2-\frac{1}{t}})$. This bound is sharp for $t=2,3$ and known matching lower bound constructions use geometric bipartite intersection graphs~\cite{Brown66}. The question whether this bound is tight for $t \geq 4$ is widely open.

	In recent years, numerous works obtained improved bounds on the number of edges in various algebraic and geometric settings (e.g.,~\cite{AK25,BCS+21,COZ25,CH23,Do19,FPSSZ17,FK21,HMTS25,JP20,KellerS24,MST24,Tomon23,TZ21}). Several of these works studied Zarankiewicz's problem for \emph{intersection graphs} of two families of geometric objects. In such a graph $G(A_1,A_2)$, the vertices are two families $A_1,A_2$ of geometric objects, and for $a_1 \in A_1$ and $a_2 \in A_2$, $(a_1,a_2)$ is an edge of $G$ if $a_1 \cap a_2 \neq \emptyset$. In particular, Chan and Har-Peled~\cite{CH23} obtained an $O(tn(\frac{\log n}{\log \log n})^{d-1})$ bound for intersection graphs of points vs.~axis-parallel boxes in $\mathbb{R}^d$, improving upon results of~\cite{BCS+21,TZ21}, and observed that a matching lower bound construction appears in a classical paper of Chazelle~\cite{Chazelle90}. They also obtained an $O(tn \log \log n)$ bound for intersection graphs of points vs.~pseudo-discs in the plane.
	Keller and Smorodinsky~\cite{KellerS24} proved a bound of $O_t(n\frac{\log n}{\log \log n})$ for intersection graphs of two families of axis-parallel rectangles in the plane and a bound of $O(t^6 n)$ for intersection graphs of two families of pseudo-discs. Both bounds are sharp, up to the dependence on $t$. Hunter, Milojevi\'{c}, Tomon and Sudakov~\cite{HMTS25} obtained the optimal bound $O(tn)$ for intersection graphs of two families of pseudo-discs.  
	
	\medskip \noindent \textbf{Zarankiewicz's problem for hypergraphs.} The `hypergraph analogue' of Zarankiewicz's problem asks for the maximum number of hyperedges in an $r$-uniform hypergraph on $n$ vertices that does not contain $K^r_{t,t,\ldots,t}$ (i.e., the complete $r$-partite $r$-uniform hypergraph with all parts having size $t$) as a subhypergraph. This question was raised in 1964 by Erd\H{o}s~\cite{Erdos64}, who obtained an upper bound of $O(n^{r-\frac{1}{t^{r-1}}})$ and showed that this bound is essentially optimal for general hypergraphs. Note that $\Omega(n^{r-1})$ is a trivial lower bound for this problem, as the complete $r$-partite graph $K_{n,n,\ldots,n,t-1}^r$ has $(t-1)n^{r-1}$ edges and is clearly $K_{t,t,\ldots,t}^r$-free.
	
	In recent years, a number of works obtained improvements of the bound of Erd\H{o}s under various algebraic assumptions on the hypergraph. Do~\cite{Do18} obtained a bound of the form $O_{D,d,t,r}(n^{r-\alpha})$ for semialgebraic hypergraphs in $\Re^d$, where $\alpha=\alpha(r,d)<1$ and $D$ is the description complexity. Improved bounds in the same setting were later obtained by Do~\cite{Do19} and by Tidor and Yu~\cite{TY24}. In particular, the results of~\cite{TY24} yield a bound of the form $O_{t,r}(n^{r-\frac{r}{3r-2}})$ for $r$-partite intersection hypergraphs of $r$ families of discs in the plane. Tong~\cite{Tong24} generalized the results of~\cite{Do18} to classes of hypergraphs that satisfy the distal regularity lemma. Basit, Chernikov, Starchenko, Tao, and Tran~\cite{BCS+21} obtained an improved bound for semilinear hypergraphs. In particular, their technique yields the bound $O_{d,t,r}(n^{r-1} (\log n)^{d(2^{r-1}-1)})$, for $r$-partite intersection hypergraphs of $r$ families of axis-parallel boxes in $\Re^d$ (see formal definition below). 
	
	All these improved bounds use in a crucial way the algebraic structure of the hypergraph. No improvements of the bound of Erd\H{o}s in non-algebraic settings are known.

	\subsection{Our results}
	
	In this paper we study Zarankiewicz's problem in intersection hypergraphs of families of geometric objects. In such an $r$-uniform $r$-partite hypergraph $H(A_1,A_2,\ldots,A_r)$, the vertices are $r$ families $A_1,A_2,\ldots,A_r$ of geometric objects, and for $a_1 \in A_1, \ldots, a_r \in A_r$, $(a_1,a_2,\ldots,a_r)$ is a hyperedge of $H$ if $a_1 \cap \ldots \cap a_r \neq \emptyset$.
	
	\medskip \noindent \textbf{Intersection hypergraphs of axis-parallel boxes.} Our first main result is the following upper bound for Zarankiewicz's problem for $r$-uniform intersection hypergraphs of axis-parallel boxes in $\Re^d$.\footnote{We note that in the preliminary version of this paper, published at SoCG 2025~\cite{CKS25}, the term $t^2$ appeared in the bound instead of $t$. The improvement to the optimal term $t$ stems from using the recent result of Chalermsook, Orgo and Zarsav~\cite{COZ25} for Zarankiewicz's problem for intersection graphs of vertical and horizontal segments.} 
	
	
	\begin{theorem}
		\label{thm:boxes} 
		Let $r,t,d \geq 2$, and let $H$ be the intersection hypergraph of families 
		$A_1, \ldots, A_r$ of $n$ axis-parallel boxes in general position in $\Re^d$. If $H$ is $K_{t,\ldots,t}^r$-free, then $$|\E(H)|=O_{r,d} \left(t n^{r-1} \left(\frac{\log n}{\log \log n} \right)^{d-1} \right).$$ 
	\end{theorem}
	This result, which improves the aforementioned result of Basit et al.~\cite{BCS+21} by a factor of about $(\log n)^{d(2^{r-1}-2)}$, is sharp, in terms of both $n$ and $t$. Indeed, as was mentioned above, Chan and Har-Peled~\cite[Appendix~B]{CH23} showed that for $r=2$, a lower bound of $\Omega_d(tn( \frac{\log n}{\log \log n})^{d-1})$ follows from an old result of Chazelle~\cite{Chazelle90}. This bound trivially yields an  $\Omega_d(tn^{r-1} ( \frac{\log n}{\log \log n})^{d-1})$ lower bound for $r$ families, by adding $r-2$ families of $n$ large boxes that contain all the boxes of the two first families. 
	
	Our proof exploits geometric properties of axis-parallel boxes and builds upon the upper bounds for Zarankiewicz's problem for intersection graphs of two families of axis-parallel boxes, obtained in~\cite{CH23} for intersections of points and boxes in $\Re^d$, in~\cite{KellerS24} for intersections of two families of rectangles in $\Re^2$, and in~\cite{COZ25,HMTS25} for intersections of families of horizontal and vertical segments in $\Re^2$.  
	
	It is somewhat surprising that the power of $\log n$ in the bound does not depend on $r$. Showing this is the most complex part of our proof, which uses biclique covers~\cite{CES83} and a non-standard inductive argument.
	
	\medskip \noindent \textbf{Intersection hypergraphs of pseudo-discs.}
	Our second main result concerns  intersection hypergraphs of pseudo-discs. A family of simple Jordan regions in the plane is called \emph{a family of pseudo-discs} if the boundaries of every two regions intersect at most twice. For technical reasons, it is convenient to assume that the pseudo-discs are $y$-monotone, namely, that the intersection of any vertical line with a region from the family is either empty or an interval. In addition, we assume that the pseudo-discs are in general position, namely, no three boundaries intersect in a point.\footnote{We note that in the preliminary version of this paper, published at SoCG 2025~\cite{CKS25}, the term $t^6$ appeared in the bound instead of $t$. The improvement to the optimal term $t$ stems from using the recent result of Hunter, Milojevi\'{c}, Tomon and Sudakov~\cite{HMTS25} for Zarankiewicz's problem for intersection graphs of pseudo-discs.}
	\begin{theorem}
		\label{thm:pd} 
		Let $r,t \geq 2$, and let $H$ be the intersection hypergraph of families $A_1, \ldots, A_r$ of $n$ $y$-monotone pseudo-discs in general position in the plane. If $H$ is $K_{t,\ldots,t}^r$-free, then $|\E(H)|=O_r(t n^{r-1} (\log n)^{r-2} )$.
	\end{theorem}
	This result is obviously sharp up to a factor of $O((\log n)^{r-2})$. Our proof builds upon the upper bound for Zarankiewicz's problem for intersection graphs of two families of pseudo-discs obtained in~\cite{HMTS25}, and uses geometric properties of pseudo-discs and shallow cuttings~\cite{Mat92}. As general intersection hypergraphs of pseudo-discs have no algebraic structure, no improvement of the 60-year-old $O(n^{r-\frac{1}{t^{r-1}}})$ bound of Erd\H{o}s~\cite{Erdos64} was known in this setting. Furthermore, in the special case of discs whose semialgebraic structure allows applying the previous algebraic works, our result improves over the best known previous result of Tidor and Yu~\cite{TY24} by a factor of $\tilde{\Omega}(n^{\frac{2r-2}{3r-2}})$.
	
	We conjecture that the right bound is $O_r(tn^{r-1})$. As we show below, such a bound would follow if one can prove an optimal bound for the lopsided Zarankiewicz's problem for intersection graphs of pseudo-discs. An optimal bound for the symmetric version of this problem was recently obtained by Hunter, Milojevi\'{c}, Tomon and Sudakov~\cite[Thm.~1.6]{HMTS25}. 
	
	\medskip \noindent \textbf{Organization of the paper.} In Section~\ref{sec:box} we present the proof of Theorem~\ref{thm:boxes}. In Section~\ref{sec:pd} we present the proof of Theorem~\ref{thm:pd}. In the appendix we prove several technical lemmas that are used in the proofs of Theorem~\ref{thm:boxes}.

	\section{Intersection Hypergraphs of Axis-Parallel Boxes}
	\label{sec:box}
	
	
	In this section we prove our new bound for Zarankiewicz's problem for intersection hypergraphs of axis-parallel boxes in $\Re^d$ -- namely, Theorem~\ref{thm:boxes}. First, in Section \ref{sec:lopdidedbox} we present a bound for the \emph{lopsided} Zarankiewicz problem for intersection \emph{graphs} of two families of boxes. 
	Then, in Section \ref{sec:boxes_general} we use the lopsided version to handle
	$r$-partite intersection hypergraphs of boxes. 
	
	\subsection{Lopsided version of Zarankiewicz's problem for two families of boxes} \label{sec:lopdidedbox}
	
	In this subsection we bound the number of edges in a $K_{t,gt}$-free intersection graph of two families of axis-parallel boxes in $\Re^d$. We prove the following:
	\begin{theorem}
	\label{cons:lopsided}
		Let $A,B$ be two multisets of boxes in $\Re^d$ where $|A|=n,|B|=m$ and $m=\mathrm{poly}(n)$. If their intersection graph $G_{A,B}$ is $K_{t,gt}$-free ($t$ on the side of $A$), then 
		\[
		|E(G_{A,B})| \leq O_d(gtn f_{d}(n) +  tm f_{d}(m)),
		\]
		where $f_{d}(n)=(\frac{\log n}{\log \log n})^{d-1}$.
	\end{theorem}
	Theorem~\ref{cons:lopsided} is a generalization and an improvement of the bound on intersections between two families of axis-parallel rectangles in $\Re^2$ proved in~\cite[Theorem~1.7]{KellerS24}. Compared to the result of~\cite{KellerS24}, Theorem~\ref{cons:lopsided} applies in $\Re^d$ for all $d \geq 2$, is not restricted to the symmetric case of $K_{t,t}$, and has a better (and optimal) dependence on $t$. 

\medskip

The proof of Theorem~\ref{cons:lopsided} is divided into two main cases, based on the following observation. Any intersection between two axis-parallel boxes belongs to one of two types:
\begin{enumerate}
	\item \emph{Vertex containment intersections}, in which a vertex of one box is contained in the other box;
	
	\item \emph{Facet intersections}, in which a facet of one box intersects a facet of the other box.
\end{enumerate}
We bound each type of intersections with a different divide-and-conquer argument, as presented below.

\subsubsection{Bounding vertex containment intersections}

Due to the lack of symmetry between the two sides, we have to consider separately two types of intersections:
\begin{itemize}
	\item Intersections in which a vertex of a box in $A$ is contained in a box in $B$;
	
	\item Intersections in which a vertex of a box in $B$ is contained in a box in $A$.
\end{itemize}
For each type, the problem reduces to bounding the number of edges in $K_{t,gt}$-free intersection graphs of points and axis-parallel boxes. Indeed, for the first case we can bound the number of intersections by $|E(G_{A',B})|$, where $A'$ is the set of vertices of the boxes in $A$. The second case can be handled similarly. We prove the following proposition, which is a lopsided version of the bound of~\cite[Theorem~4.5]{CH23} on the number of edges in $K_{t,t}$-free bipartite intersection graphs of points and axis-parallel boxes in $\Re^d$. 
\begin{proposition}\label{Prop:PtsBoxes}
	Let $A$ be a multiset of points in $\Re^d$ and let $B$ be a multiset of axis-parallel boxes, where $|A|=n$ and $|B|=m$. Then for any $\epsilon>0$,
	\begin{enumerate}
		\item If the bipartite intersection graph $G_{A,B}$ is $K_{t,gt}$-free ($t$ on the side of $A$) then it has \newline $O_{\epsilon} \left(gt n (\frac{\log n}{\log \log n})^{d-1}+ tm (\frac{\log n}{\log \log n})^{d-2+\epsilon}\right)$ edges.
		
		\item If the bipartite intersection graph of $G_{A,B}$ is $K_{gt,t}$-free ($t$ on the side of $B$) then it has $O_{\epsilon} \left(t n (\frac{\log n}{\log \log n})^{d-1}+ gt m (\frac{\log n}{\log \log n})^{d-2+\epsilon}\right)$ edges.
	\end{enumerate}
\end{proposition}
As the proof of Proposition~\ref{Prop:PtsBoxes} is somewhat lengthy and technical, we present it in Appendix~\ref{app:rect}.
 
\subsubsection{Bounding facet intersections}

We claim that for a large $n$, the number of facet intersections is significantly smaller than the number of vertex containment intersections. 
\begin{proposition}\label{Prop:Facet}
	Let $A,B$ be multisets of axis-parallel boxes in $\Re^d$, where $|A|=n$ and $|B|=m$. If the bipartite intersection graph $G_{A,B}$ is $K_{t,gt}$-free ($t$ on the side of $A$) then the number of facet intersections between a box in $A$ and a box in $B$ is
	$$O_d \left(gt n (\log n)^{d-2} +tm (\log n)^{d-2} \right).$$
\end{proposition}

We use the following proposition.
\begin{proposition}\label{Prop:HorVer}
	Let $A$ be a multiset of horizontal segments in $\Re^2$ and let $B$ be a multiset of vertical segments in $\Re^2$, where $|A|=n$ and $|B|=m$. Let $g>0$. If the bipartite intersection graph $G_{A,B}$ is $K_{t,gt}$-free (where the $t$ is on the side of $A$) then it has at most $27(t-1)m+27(gt-1)n$ edges.    
\end{proposition}
The symmetric version of the proposition (i.e., the case $g=1$) was recently obtained independently in two papers, using entirely different methods. The first, by Chalermsook, Orgo and Zarsav~\cite{COZ25}, proves the proposition directly by a beautiful charging argument. The second, by Hunter, Milojevi\'{c}, Tomon and Sudakov~\cite{HMTS25}, derives the proposition (up to replacing $27$ by another constant) as a byproduct of a general powerful method which allows leveraging linear bounds for Zarankiewicz's problem w.r.t.~containment of $K_{2,2}$ to bounds w.r.t.~containment of $K_{t,t}$ which are linear in both $n$ and $t$.

As the proof of Proposition~\ref{Prop:HorVer} is a straightforward generalization of the argument of~\cite[Theorem~3]{COZ25}, we sketch it in Appendix~\ref{app:rect} for the sake of completeness.
	
\begin{proof}[Proof of Proposition~\ref{Prop:Facet}]
	The proof is by induction on $d$. The base case $d=2$ amounts to bounding the number of edges in a $K_{2t,2gt}$-free bipartite intersection graph of a multiset $A$ of horizontal segments in $\Re^2$ and a multiset $B$ of vertical segments in $\Re^2$.
	Indeed, each `facet intersection' in the plane is either an intersection between a horizontal edge of a rectangle in $A$ with a vertical edge of a rectangle in $B$, or vice versa. We can bound the number of facet intersections of the first type by $|E(G_{A',B'})|$, where $A'$ is the multiset of horizontal edges of rectangles in $A$ and $B'$ is the multiset of vertical edges of rectangles in $B$. Note that this intersection graph is $K_{2t,2gt}$-free, as otherwise, $G_{A,B}$ would contain $K_{t,gt}$. The second type can be handled similarly. Hence, Proposition~\ref{Prop:HorVer} yields a bound of $O(gtn +tm)$ in this setting, which proves the induction basis. 
	
	\medskip



	In the induction step, we assume that for $(d-1)$-dimensional boxes, the number of facet intersections is bounded by 
	\[
	O_d(gt n (\log n)^{d-3} +tm (\log n)^{d-3}).
	\]
	We bound the number of facet intersections between $x \in A$ and $y \in B$, such that the intersecting facets are othogonal to the $(d-1)$'th and the $d$'th axis, respectively. A bound on the total number of facet intersections clearly follows by multiplying with $d^2$. In the rest of the proof, we call such special intersections `good facet intersections', or just `intersections' for brevity. We make sure that in the dimension reductions in the induction process presented below, every time the first coordinate is the one that is removed, and thus, the good facet intersections are not affected.
	
	We use the following auxiliary notion. Let $I_d(n,m)$ be the maximum number of good facet intersections between multisets $A,B$ of axis-parallel boxes inside a `vertical strip' $\mathcal{U}=\{x \in \Re^d:u_L<x_1<u_R\}$ of $\Re^d$, where:
	\begin{enumerate}
		\item Each box in $A,B$ has either half of its vertices or all of its vertices in $\mathcal{U}$;
		
		\item The total number of vertices of boxes in $A$ (resp., boxes in $B$) inside $\mathcal{U}$ is $n \cdot 2^{d-1}$ (resp., $m \cdot 2^{d-1}$);
		
		\item The bipartite intersection graph of $A,B$ is $K_{t,gt}$-free.
	\end{enumerate}
	We shall prove that 
	\[
	I_d(n,m) \leq O_d(gt n (\log n)^{d-2} +tm (\log n)^{d-2}).
	\] 
	This clearly implies the assertion of the proposition, as by considering a vertical strip $\mathcal{U}$ that fully contains all boxes in $A$ and $B$ (where $|A|=n$ and $|B|=m$), we get that the number of good facet intersections between $A$ and $B$ is at most $I_d(2n,2m)$. (Note that we neglect factors of $O_d(1)$). 
	
	
	Let $A,B$ be multisets of boxes that satisfy assumptions~(1)--(3) with respect to a strip $\mathcal{U} \subset \Re^d$. We divide $\mathcal{U}$ into two vertical sub-strips $\sigma_1=\{x \in \Re^d: u_L<x_1<u'\}$ and $\sigma_2=\{x \in \Re^d:u'<x_1<u_R\}$, such that each sub-strip contains $n \cdot 2^{d-2}$ vertices of boxes in $A$. For $i=1,2$, we denote by $A_i$ (resp., $B_i$) the boxes in $A$ (resp., $B$) that have at least one vertex in $\sigma_i$, and by $A'_i$ (resp., $B'_i$) the boxes in $A$ (resp., $B$) that intersect $\sigma_i$ but do not have vertices in it. Note that $A'_1 \subset A_2, A'_2 \subset A_1$, and similarly for $B$. We also denote the number of vertices of boxes in $B$ contained in $\sigma_i$ by $2^{d-1} \cdot m_i$.
	
	The number of good facet intersections in $\mathcal{U}$ between boxes in $A$ and boxes in $B$ is at most 
	\begin{equation}\label{Eq:Bound-boxes1}
	I(A_1,B_1)+I(A_2,B_2)+I(A_1,B'_1)+I(A'_2,B_2)+I(A_2,B'_2)+I(A'_1,B_1),
	\end{equation}
	where $I(X,Y)$ denotes the number of good facet intersections between an element of $X$ and an element of $Y$.
	
	Indeed, $I(A_1,B_1)$ (resp., $I(A_2,B_2)$) counts intersections between pairs of boxes that have at least one vertex in $\sigma_1$ (resp., $\sigma_2$). $I(A_1,B'_1)+I(A'_2,B_2)$ upper bounds the number of intersections between a box in $A$ that has a vertex in $\sigma_1$ and a box in $B$ that has a vertex in $\sigma_2$. $I(A_2,B'_2)+I(A'_1,B_1)$ upper bounds the number of intersections between a box in $A$ that has a vertex in $\sigma_2$ and a box in $B$ that has a vertex in $\sigma_1$.
	
	By the definitions, we have $I(A_1,B_1)\leq I_d(\frac{n}{2},m_1)$ (where $\sigma_1$ is taken as the vertical strip instead of $\mathcal{U}$). Similarly, $I(A_2,B_2) \leq I_d(\frac{n}{2},m_2)$. 
	
	To handle the other types of intersections, we observe that a box in $A_1$ intersects a box in $B'_1$ if and only if their projections on the hyperplane $\mathcal{H}=\{x \in \Re^d: x_1=u'\}$ (which are $(d-1)$-dimenstional boxes) intersect. Let $\bar{A}_1$ and $\bar{B}'_1$ denote the corresponding families of projections. We have $|\bar{A}_1|=|A_1|=\frac{n}{2}$ and $|\bar{B}'_1|=|B'_1|\leq m_2$. The multisets  $\bar{A}_1$ and $\bar{B}'_1$ are multisets of axis-parallel boxes in $\Re^{d-1}$ whose bipartite intersection graph is $K_{t,gt}$-free. Hence, by the induction hypothesis we have
	\begin{align*}
	\begin{split}
		I(A_1,B'_1)=I(\bar{A}_1,\bar{B}'_1) &\leq 
			O_d \left(gt \tfrac{n}{2} \left(\log \tfrac{n}{2}\right)^{d-3} +tm_2 \left(\log \tfrac{n}{2}\right)^{d-3}\right) \\ 
			&\leq O_d\left(gt \tfrac{n}{2} \left(\log n\right)^{d-3} +tm_2 \left(\log n\right)^{d-3}\right).
	\end{split}
	\end{align*}
	Applying the same argument to $I(A'_2,B_2), I(A'_1,B_1)$ and $I(A_2,B'_2)$, we get
	\begin{align*}
		\begin{split}
		I(A_1,B'_1)&+I(A'_2,B_2)+I(A_2,B'_2)+I(A'_1,B_1) \\ &\leq  
		O_d(4gt \tfrac{n}{2} (\log n)^{d-3} +2tm_2 (\log n)^{d-3} + 2tm_1 (\log n)^{d-3}) \\
		&= O_d(2gt n (\log n)^{d-3}+ 2tm (\log n)^{d-3}).
		\end{split}
	\end{align*}
	Combining this with the bounds on $I(A_1,B_1)$ and $I(A_2,B_2)$ and substituting to~\eqref{Eq:Bound-boxes1}, we obtain the recursive formula 
	\[
	I_d(n,m) \leq \max_{m_1+m_2=m} \left(I_d(\tfrac{n}{2},m_1)+I_d(\tfrac{n}{2},m_2)+
	O_d(2gt n (\log n)^{d-3}+ 2tm (\log n)^{d-3}) \right),
	\]
	which solves to 
	\[
	I_d(n,m) \leq \log n \cdot O_d(gt n (\log n)^{d-3}+ tm (\log n)^{d-3}) \leq O_d(gt n (\log n)^{d-2}+ tm (\log n)^{d-2}).
	\]
	This completes the proof.
\end{proof}	
	
\subsubsection{Completing the proof of Theorem~\ref{cons:lopsided}}

Now we are ready to wrap up the proof of Theorem~\ref{cons:lopsided}.

\begin{proof}[Proof of Theorem~\ref{cons:lopsided}]
	Let $A,B$ be multisets of axis-parallel boxes in $\Re^d$ that satisfy the assumptions of the theorem. As written above, the intersections between $A$ and $B$ can be divided into \emph{vertex containment intersections} and \emph{facet intersections}. 
	
	To bound the number of vertex containment intersections, we apply Proposition~\ref{Prop:PtsBoxes}, with $\epsilon=\frac{1}{2}$. Specifically, we use Proposition~\ref{Prop:PtsBoxes}(1) to bound the number of intersections in which a vertex of $a \in A$ is contained in a box $b \in B$. Furthermore, we use Proposition~\ref{Prop:PtsBoxes}(2) \emph{with the roles of $n$ and $m$ reversed} to bound the number of intersections in which a vertex of $b \in B$ is contained in a box $a \in A$.   We get that the number of vertex containment intersections is at most
	$$O_d\left(gt n (\tfrac{\log n}{\log \log n})^{d-1}+ tm (\tfrac{\log n}{\log \log n})^{d-1.5} + t m (\tfrac{\log m}{\log \log m})^{d-1}+ gt n (\tfrac{\log m}{\log \log m})^{d-1.5}\right).$$ 
	By Proposition~\ref{Prop:Facet}, the number of facet intersections is at most 
	$$O_d(gt n (\log n)^{d-2}+ tm (\log n)^{d-2}).$$ 
	Note that for a large $n$, we have $(\log n)^{d-2} = o((\frac{\log n}{\log \log n})^{d-1.5})$, and hence, the number of facet intersections is dominated by the number of vertex containment intersections. 
	
	Finally, since $m=\mathrm{poly(n)}$, the terms $tm (\frac{\log n}{\log \log n})^{d-1.5}$ and $gt n (\frac{\log m}{\log \log m})^{d-1.5}$ are dominated by the two other terms. Therefore, we have
	\[
	|E(G_{A,B})| \leq O_d\left(gt n (\tfrac{\log n}{\log \log n})^{d-1}+ t m (\tfrac{\log m}{\log \log m})^{d-1} \right) = O_d(gtn f_{d}(n)+tm f_{d}(m)),
	\]  	  
	as asserted. This completes the proof.
\end{proof}

	\subsection{Zarankiewicz problem for $r$-partite intersection hypergraph of boxes} \label{sec:boxes_general}
	
	Using Theorem~\ref{cons:lopsided}, we can relatively easily obtain Proposition~\ref{thm:boxes_weaker} below, which is a weaker version of Theorem~\ref{thm:boxes}. Since the proof of the weaker result may serve as a good introduction to the proof of Theorem~\ref{thm:boxes}, we present it in Section~\ref{sec:boxes_weak}, and then we pass to the proof of Theorem~\ref{thm:boxes} in Section \ref{sec:boxes}. 
	
	\subsubsection{Proof of a weaker variant of Theorem \ref{thm:boxes}} \label{sec:boxes_weak}
	
	
	
	
We prove the following.	
	\begin{proposition}\label{thm:boxes_weaker}
		Let $A_1,A_2,\ldots,A_r$ be families of $n$ axis-parallel boxes in $\Re^d$, and let $H$ be their $r$-partite intersection hypergraph. 
		If $H$ is $K_{t,t,\ldots,t}^r$-free, then $|\E(H)|=O_r\left(t \left(n f_{d}(n)\right)^{r-1} \right)$.
	\end{proposition}
	
	
	\begin{proof}
	The proof is by induction on $r$.
	
	\medskip \noindent \textbf{Induction basis: $r=3$.}
	Let $A_1=A,A_2=B,A_3=C$.
	Let $AB$ be the following multiset of axis-parallel boxes: $AB=\{a\cap b:a \in A, b\in B, a\cap b \neq \emptyset\}$. (Note that the intersection of two axis-parallel boxes is indeed an axis-parallel box). Let $G$ be the bipartite intersection graph of the families $C$ and $AB$. It is clear that $|E(G)|=|\E(H)|$, since there is a clear one-to-one correspondence between edges of $G$ and hyperedges of $H$.
	
	We claim that for a sufficiently large constant $M$, the graph $G$ is $K_{t,M tn f_{d}(n) }$-free. Indeed, assume to the contrary that $G$ contains a copy of $K_{t,Mtn  f_{d}(n)}$, for a `large' $M$. This means that there exist $t$ boxes $c_1,c_2,\ldots,c_t \in C$ which all have a non-empty intersection with certain 
	$Mtn  f_{d}(n)$ axis-parallel boxes of the form $a_i \cap b_i$, with $a_i \in A$, $b_i \in B$. Denote by $A'$ the set of all $a_i \in A$ that participate in such intersections, and by $B'$ the set of all $b_i \in B$ that participate in such intersections. Let $G'$ be the bipartite intersection graph of $A',B'$. We have $|E(G')|\geq Mtn  f_{d}(n)$, and hence, by Theorem~\ref{cons:lopsided} (applied with $m=n$ and $g=1$), it contains a $K_{t,t}$ (using the assumption that $M$ is sufficiently large). This means that there exist $a_1,a_2,\ldots,a_t \in A$ and $b_1,b_2,\ldots,b_t \in B$ such that for all $1 \leq i,j \leq t$, we have $a_i \cap b_j \neq \emptyset$, and both $a_i$ and $b_j$ have a non-empty intersection with each of $c_1,\ldots,c_t$ (since they participate in pairs whose intersection with each of $c_1,\ldots,c_t$ is non-empty). As any three pairwise intersecting axis-parallel boxes have a non-empty intersection, this implies that $a_i \cap b_j \cap c_k \neq \emptyset$ for all $1 \leq i,j,k \leq t$, and consequently, $H$ contains a $K_{t,t,t}$, a contradiction.
	
	We have thus concluded that $G$ is $K_{t,M tn  f_{d}(n)}$ free, for a sufficiently large constant $M$. By Theorem~\ref{cons:lopsided}, applied with $n$, $m=n^2$, and $g=M  n  f_{d}(n)$, this implies that 
	\[
	|\E(H)|=|E(G)| \leq O\left(t  \left(n  f_{d}(n)\right)^2 \right).
	\]
	This completes the proof of the induction basis.

	\medskip \noindent \textbf{Induction step: $r>3$.} Now, assume we proved the assertion for $r-1$ and consider families $A_1,A_2,\ldots,A_r$ of $n$ axis-parallel boxes. 
	
	Let $A_1A_2 \ldots A_{r-1}$ be the multiset of axis-parallel boxes 
	\[
	A_1A_2 \ldots A_{r-1}=\{a_1\cap a_2\cap \ldots \cap a_{r-1}:a_1 \in A_1, \ldots, a_{r-1}\in A_{r-1}, a_1\cap \ldots \cap a_{r-1} \neq \emptyset\}.
	\]
	Let $G$ be the bipartite intersection graph of the families $A_r$ and $A_1A_2 \ldots A_{r-1}$. It is clear that $|E(G)|=|\E(H)|$.
	
	We claim that there exists $M_{r-1}$ such that $G$ is $K_{t,M_{r-1} t (n f_{d}(n))^{r-2}}$ free. Indeed, assume on the contrary that $G$ contains a copy of $K_{t,M_{r-1} t (n  f_{d}(n))^{r-2}}$, for a sufficiently large $M_{r-1}$. This means that there exist $a_{r_1},a_{r_2},\ldots,a_{r_t} \in A_r$ which all have non-empty intersection with certain 
	$M_{r-1} t (n  f_{d}(n))^{r-2}$ axis-parallel boxes of the form $a_{1_j} \cap \ldots \cap a_{(r-1)_j}$, with $a_{i_j} \in A_i$. Denote by $A_i'$ the set of all $a_{i_j} \in A_i$ that participate in such intersections, and let $H'$ be the $(r-1)$-partite intersection hypergraph of $A_1',\ldots,A'_{r-1}$. We have 
	\[
	|\E(H')|\geq M_{r-1} t  (n  f_{d}(n))^{r-2},
	\]
	and hence, by the induction hypothesis, it contains a $K_{t,t,\ldots,t}^{r-1}$, assuming $M_{r-1}$ is sufficiently large. This means that for any $1 \leq j \leq r-1$, there exist $a_{j_1},a_{j_2},\ldots,a_{j_t} \in A_j$ such that any intersection of the form $a_{1_{j,1}} \cap a_{2_{j,2}} \cap \ldots \cap a_{{r-1}_{j,r-1}}$ is non-empty, and all $a_{i_j}$'s have a non-empty intersection with each of $a_{r_1},\ldots,a_{r_t}$ (since they participate in $(r-1)$-tuples whose intersection with each of $a_{r_1},\ldots,a_{r_t}$ is non-empty). As any set of pairwise intersecting axis-parallel boxes has a non-empty intersection, this implies that $H$ contains a $K_{t,t,\ldots,t}^r$, a contradiction.
	
	We have thus concluded that $G$ is $K_{t,M_{r-1} t (n  f_{d}(n))^{r-2}}$ free, for a sufficiently large constant $M$. By Theorem~\ref{cons:lopsided}, applied with $n$, $n^{r-1}$, and $g=M_{r-1} t (n  f_{d}(n))^{r-2}$, this implies that 
	\[
	|\E(H)|=|E(G)| \leq O_r\left(t \left(n  f_{d}(n)\right)^{r-1} \right),
	\]
	as asserted. This completes the inductive proof (with the $O_r(\cdot)$ dependence exponential in $r$).
\end{proof}

	Note that in this proof, we used the fact that axis-parallel boxes in $\Re^d$ have \emph{Helly number 2}, namely, that any pairwise intersecting family of axis-parallel boxes has a non-empty intersection. Without this property, the existence of $K_{t,t}$ in $G'$ does not necessarily imply the existence of $K_{t,t,t}$ in $H$.
	
	
	\subsubsection{A modified biclique cover theorem for axis-parallel boxes}\label{sec:biclique}
	
	In Proposition~\ref{thm:boxes_weaker}, the polylogarithmic factor in the upper bound increases with $r$, since at each step of the inductive process we apply Theorem~\ref{cons:lopsided} and `pay' another $f_{d}(n)$-factor. In order to prove the stronger Theorem~\ref{thm:boxes}, we have to avoid this dependency on $r$. To this end,  
	we use the following modification of the \emph{biclique cover theorem} (see~\cite{CES83}) for axis-parallel boxes, that may be of independent interest. We note that the relation between Zarankiewicz's problem and biclique covers was already observed in~\cite{CY25,CH23,Do19}.
	\begin{definition}
		A biclique cover of a graph $G=(V,E)$ is a collection of pairs of vertex subsets $\{(A_1,B_1),\ldots,(A_l,B_l)\} $ such that $E=\bigcup_{i=1}^l (A_i \times B_i)$. The size of the cover is $\Sigma_{i=1}^l \left(|A_i| + |B_i|\right)$.
	\end{definition}
	
	\begin{lemma}\label{lem:biclique}
		Let $n,m,b \in \mathbb{N}$ be such that $b \leq \min(n,m)$, and let $A,B$ be families of axis-parallel boxes in $\Re^d$, with $|A|=n$ and $|B|=m$. Then the bipartite intersection graph of $A,B$ can be partitioned into a union of $O(b \log^{d-1} b)$ bicliques (with no restriction on their size) and $O(b \log^{d-1} b)$ ``partial bicliques'' (namely, subgraphs of bicliques), each of size at most $\frac{n}{b} \cdot \frac{m}{b}$. 
	\end{lemma}

It is clearly sufficient to prove that the following holds for any parameters $p,q \leq \min(n,m)$. 
\begin{claim}\label{Claim:Biclique}
	Let $n,m,p,q \in \mathbb{N}$ be such that $p,q \leq \min(n,m)$, and let $A,B$ be families of axis-parallel boxes in $\Re^d$, with $|A|=n$ and $|B|=m$. Then the bipartite intersection graph of $A,B$ can be partitioned into a union of $O((\frac{n}{p}+\frac{m}{q})\log^{d-1}(\frac{n}{p}))$ bicliques (with no restriction on their size) and $O((\frac{n}{p}+\frac{m}{q})\log^{d-1}(\frac{n}{p}))$ ``partial bicliques'' (namely, subgraphs of bicliques), each of size at most $p \cdot q$. 
\end{claim}	
Indeed, applying the claim with $p=\frac{n}{b}$ and $q=\frac{m}{b}$ yields the assertion of the lemma.

\begin{proof}[Proof of Claim~\ref{Claim:Biclique}]
	The proof uses a standard divide-and-conquer argument, like in the proof of Proposition~\ref{Prop:Facet}. We use the following auxiliary notion.
	
	Let $N_d(n,m)$ be the maximal total number of bicliques (with no restriction on their size) and partial bicliques, each of size at most $p\cdot q$, needed to cover the intersection graph of two multisets $A,B$ of axis-parallel boxes inside a `vertical strip' $\mathcal{U}=\{x \in \Re^d:u_L<x_1<u_R\}$ of $\Re^d$, where:
	\begin{enumerate}
		\item Each box in $A,B$ has either half of its vertices or all of its vertices in $\mathcal{U}$;
		
		\item The total number of vertices of boxes in $A$ (resp., boxes in $B$) inside $\mathcal{U}$ is $n \cdot 2^{d-1}$ (resp., $m \cdot 2^{d-1}$).
	\end{enumerate}
	We shall prove that 
	\[
	N_d(n,m) \leq O \left(\left(\tfrac{n}{p}+\tfrac{m}{q}\right)\log^{d-1}\left(\tfrac{n}{p}\right) \right).
	\] 
	This clearly implies the assertion of the claim, as by considering a vertical strip $\mathcal{U}$ that fully contains all boxes in $A$ and $B$ (where $|A|=n$ and $|B|=m$), we get that the number of bicliques (with no restriction on their size) and partial bicliques, each of size at most $p\cdot q$, needed to cover the intersection graph of $A,B$ is at most $N_d(2n,2m)$. 
	
	The proof is by induction on $n$ and $d$. 
	
	\medskip \noindent \textbf{Induction base.} The base cases are $n \leq p$ and $d=1$. If $n \leq p$ then for any $d,m$, we have
	$N_d(n,m) \leq \lceil m/q\rceil$, since  we can cover the intersection graph by $\lceil m/q\rceil$ partial bicliques of size at most $p \cdot q$. 
	
	For $d=1$, $A,B$ are multisets of intervals on a line, $|A|=n,|B|=m$. We divide the line into $\lceil \frac{2n}{p} \rceil$ segments $\sigma_i$ such that each segment contains at most $p$ endpoints of intervals in $A$. Denote the number of intervals in $B$ which have an endpoint in $\sigma_i$ by $m_i$. For each $i$, all intersections between intervals of $A$ and $B$ that have an endpoint in $\sigma_i$ can be covered by  $\lceil m_i/q\rceil$ partial bicliques of size at most $p \cdot q$. Furthermore, all intersections between intervals of $A$ that intersect $\sigma_i$ (and either have an endpoint in it or not) and intervals of $B$ that intersect $\sigma_i$ but do not have an endpoint in it can be covered by a single biclique, since they form a complete bipartite graph. Similarly, all intersections between intervals of $B$ that intersect $\sigma_i$ (and either have an endpoint in it or not) and intervals of $A$ that intersect $\sigma_i$ but do not have an endpoint in it can be covered by a single biclique, Hence, all intersections between intervals in $A$ and intervals in $B$ can be covered by $O(\frac{m}{q})$ partial bicliques of size at most $p \cdot q$ and $O(\frac{n}{p})$ complete bicliques, as asserted.
	
	\medskip \noindent \textbf{Induction step.}	We assume that the claim holds for dimension $d-1$ and prove it for dimension $d$.
	We divide $\mathcal{U} \subset \Re^d$ into two vertical sub-strips $\sigma_1=\{x \in \Re^d: u_L<x_1<u'\}$ and $\sigma_2=\{x \in \Re^d:u'<x_1<u_R\}$, such that each sub-strip contains $n \cdot 2^{d-2}$ vertices of boxes in $A$. For $i=1,2$, we denote by $A_i$ (resp., $B_i$) the boxes in $A$ (resp., $B$) that have at least one vertex in $\sigma_i$, and by $A'_i$ (resp., $B'_i$) the boxes in $A$ (resp., $B$) that intersect $\sigma_i$ but do not have vertices in it. Note that $A'_1 \subset A_2, A'_2 \subset A_1$, and similarly for $B$. We also denote the number of vertices of boxes in $B$ contained in $\sigma_i$ by $2^{d-1} \cdot m_i$.
	
	By the definition of $N_d(n,m)$, the number of bicliques and partial bicliques needed to cover the intersecting pairs
	between $A_1$ and $B_1$ (resp., between $A_2$ and $B_2$) is $N_d(\frac{n}{2},m_1)$ (resp., $N_d(\frac{n}{2},m_2))$.
	
	To cover the other types of intersecting pairs, we observe that a box in $A_1$ intersects a box in $B'_1$ if and only if their projections on the hyperplane $\mathcal{H}=\{x \in \Re^d: x_1=u'\}$ (which are $(d-1)$-dimensional boxes) intersect. Let $\bar{A}_1$ and $\bar{B}'_1$ denote the corresponding families of projections.  We recursively construct bicliques and partial bicliques
	to cover the intersecting pairs between $\bar{A}_1$ and $\bar{B}'_1$ (a $(d-1)$-dimensional subproblem). By definition, the number of bicliques and partial bicliques needed for this is at most $N_{d-1}(2 \cdot \frac{n}{2},2m_2)$. (The factor 2 is needed as the number of vertices of boxes in $\bar{A}_1$ is at most $2^{d-1} \cdot \frac{n}{2} = 2^{d-2} \cdot (2 \cdot \frac{n}{2})$, and similarly for $\bar{B}'_1$). Other types can be handled similarly.
	
	Combining all types of intersections between $A$ and $B$ inside the strip $\mathcal{U}$, we obtain the recursive formula
	\begin{equation}\label{Eq:Recursion}
		N_d(n,m) \leq \max_{m_1+m_2=m} \left(N_d(\tfrac{n}{2},m_1)+N_d(\tfrac{n}{2},m_2)+
		O(N_{d-1}(2n,2m)) \right).
	\end{equation}
	We apply the recurrence until we obtain $N_d(n',m')$ for $n' \leq p$ (thus, $\lceil \log \frac{n}{p} \rceil$ steps in total) and then we apply the induction basis. Using the induction hypothesis to handle the terms $N_{d-1}(n',m')$ we encounter in the process, we obtain 
	\[
	N_d(n,m)=O\left( \left(\tfrac{n}{p}+\tfrac{m}{q}\right)\log^{d-1}\left(\tfrac{n}{p}\right)\right),
	\]
	as asserted. This completes the proof of Claim~\ref{Claim:Biclique} and of Lemma~\ref{lem:biclique}.
\end{proof}   		

\subsubsection{Proof of Theorem~\ref{thm:boxes}} \label{sec:boxes}

Now we are ready to prove Theorem~\ref{thm:boxes}. Let us restate it.
\begin{theorem}[Theorem~\ref{thm:boxes} restated]
	\label{thm:unified}	
	Let $r,d,t \geq 2$ and let $H$ be the intersection hypergraph of $r$ families $A_1, \ldots, A_r$ of $n$ axis-parallel boxes in general position in $\Re^d$. If $H$ is $K_{t,\ldots,t}^r$-free, then $|\E(H)|=O_{r,d}(t n^{r-1} f_{d}(n) )$, where 
	$f_{d}(n) = (\frac{\log n}{\log \log n})^{d-1}$.
\end{theorem}

	\medskip The motivation behind our proof of Theorem~\ref{thm:boxes} is as follows: Assume that we add an additional assumption that the $(r-1)$-partite intersection graph of the families $A_1,\ldots,A_{r-1}$ is complete. Namely, that for all $1 \leq i<j \leq r-1$, any box in $A_i$ intersects any box in $A_j$. Furthermore, assume that by applying (only once!) Theorem~\ref{cons:lopsided}, we show that the bipartite intersection graph of the multiset $A_1 A_2 \ldots A_{r-1}= \{
	a_1 \cap \ldots\cap a_{r-1} : \forall i, a_i \in A_i, \cap_{i=1}^{r-1} a_i \neq \emptyset \}$ and the family $A_r$ contains $K_{gt,t}$, for $g=n^{r-2}$. This means that some $t$ boxes from $A_r$ intersect $gt=tn^{r-2}$ boxes of the type $\{a_1 \cap \ldots\cap a_{r-1} : a_i \in A_i\}$. These $gt$ boxes involve at least $t$ boxes from each $A_i$ ($1 \leq i \leq r-1$). Thus, by our additional assumption and by the fact that axis-parallel boxes admit Helly number 2, the existence of $K_{gt,t}$ in the bipartite intersection graph, implies the existence of $K_{t, \ldots,t}^r$ in the original $r$-partite intersection hypergraph. Therefore, the suggested additional assumption enables avoiding repeated applications of Theorem~\ref{cons:lopsided}, and hence obtaining a bound with a logarithmic factor that does not increase with $r$.
		
	In order to show that we can indeed make the additional assumption described above without affecting the final bound, we define a constraints graph $G$, not to be confused with the auxiliary graph $G$ from Section \ref{sec:boxes_weak}.
		
		\begin{definition}\label{def:constraint}
			For an $r$-tuple of set families $A_1,\ldots,A_r$, the \emph{constraints graph} $G=G_{A_1,\ldots,A_r}$ is defined as follows. The vertex set of $G$ is $V(G)=\{1,2,\ldots,r\}$, and $(i,j) \in G$ if $\exists a_i \in A_i, a_j \in A_j$ such that $a_i \cap a_j = \emptyset$.
		\end{definition}
		
		This graph $G$ represents the ``distance'' of the $r$-tuple of families $A_1,\ldots,A_r$ from the desired setting in which there exists $i$ such that any box from $A_i$ intersects any box from $A_j$, for all $j \neq i$. Our goal in the proof below is to remove all the edges from $G$, except for those emanating from a single vertex $i$. This will show that one can indeed make the additional assumption described above without affecting the final bound.
		
	\begin{proof}[Proof of Theorem~\ref{thm:unified}]
		Let $T_G(n_1, \ldots,n_r)$ be the maximum number of hyperedges in a $K_{t, \ldots,t}^r$-free $r$-partite intersection hypergraph $H$ of $r$ families $A_1, \ldots,A_r$ of axis-parallel boxes in $\Re^d$, where $|A_i|=n_i$, and the constraints graph of $r$-tuple $A_1,\ldots,A_r$ is $G$. Denote $T_G(n) = T_g(n,\ldots,n)$. In order to prove the theorem, it is clearly sufficient to prove that for any constraint graph $G$, we have 
		$T_G(n) \leq O(t n^{r-1} f_{d}(n))$.
		We prove this claim by induction on $|E(G)|$. 
		
		
		
		If $E(G)=\emptyset$, then for all $i \neq j$, any $a_i \in A_i$ intersects any $a_j \in A_j$. Since $H$ is $K_{t, \ldots,t}^r$-free, this implies $n<t$ and we are done. 
		
		In the induction step, we assume that we have already proved the claim for any constraints graph with a smaller number of edges, and we now consider a constraints graph $G$. We consider three cases:
		\begin{itemize}
			\item Case A: $G$ contains two non-adjacent edges;
			
			\item Case B: $G$ is a triangle;
			
			\item Case C: $G$ is a star.
		\end{itemize} 	
		Clearly, any graph $G$ belongs to one of the three cases.
			 
	\medskip \noindent \textbf{Case A: $G$ contains two non-adjacent edges.} 
	Say these two non-adjacent edges are $(1,2)$ and $(3,4)$. Our goal now is to `remove' the edge $(1,2)$ (and later, also the edge $(3,4)$) from $G$, by partitioning the intersection graph of $A_1$ and $A_2$ into bicliques.
	To this end, we use Lemma~\ref{lem:biclique} with a large constant $b$ to partition the entire hypergraph into smaller hypergraphs, induced by a partition of the bipartite intersection of $A_1$ and $A_2$ to `full' bicliques and `partial' bicliques. We obtain the recursion: 
			\begin{equation}\label{eq:1}
				T_G(n) \leq O(b \log^{d-1} b)T_G(\tfrac{n}{b},  \tfrac{n}{b}, n, \ldots, n)+ O(b \log^{d-1} b)O(tn^{r-1} f_{d}(n)),
			\end{equation}
	where the right term comes from applying the induction hypothesis on the hypergraphs induced by full bicliques, whose constraints graph is $G \setminus \{(1,2)\}$, and the left term comes from the partial bicliques.
			
	Now, in order to bound $T_G(\tfrac{n}{b},  \tfrac{n}{b}, n, \ldots, n)$, we do the same with the sets $A_3,A_4$ and obtain
			\begin{equation}\label{eq:2}
				T_G(\tfrac{n}{b},  \tfrac{n}{b}, n, \ldots, n) \leq O(b \log^{d-1} b)T_G(\tfrac{n}{b},  \tfrac{n}{b},\tfrac{n}{b},  \tfrac{n}{b}, n, \ldots, n)+ O(b \log^{d-1} b)O(tn^{r-1} f_{d}(n)).
			\end{equation}
			Combining (\ref{eq:1}) and (\ref{eq:2}) together, we get
			\begin{equation}\label{eq:3}
				T_G(n) \leq O(b^2 \log^{2d-2} b)   \left( T_G(\tfrac{n}{b},  \tfrac{n}{b},\tfrac{n}{b},  \tfrac{n}{b}, n, \ldots, n)+  O(tn^{r-1} f_{d}(n))  \right).
			\end{equation}
			To bound the left term in the right hand side, we partition the hypergraph into $b^{r-4}$ hypergraphs which $r$ sides of size $\tfrac{n}{b}$ by partitioning each of the sets $A_5,\ldots,A_r$ arbitrarily into $b$ parts, each of size $\tfrac{n}{b}$. We obtain  
			\begin{equation}\label{eq:4}
				T_G(n) \leq O(b^2 \log^{2d-2} b)   \left(b^{r-4}T_G(\tfrac{n}{b})+ O (tn^{r-1} f_{d}(n))  \right).
			\end{equation}
			The recursion~\eqref{eq:4} solves to $	T_G(n) =O(tn^{r-1} f_{d}(n)) $, and we are done.
			
			\begin{remark}
				Note that if we had started by partitioning arbitrarily each $A_i$ into $b$ parts, then the term $O(b^{r-2} \log^{2d-2}b) T_G(\frac{n}{b})$ in (\ref{eq:4}) would have been replaced by $O(b^{r}) T_G(\frac{n}{b})$, and the recursion would solve to $\Omega(n^r)$. The use of the biclique cover enables us to split two coordinates into $b$ parts at the cost of a factor of $O(b\log^{d-1} b)$, instead of $O(b^2)$. Performing one such split allows reducing the bound to $O(tn^{r-1}\log n f_{d}(n))$, and performing two splits allows obtaining the desired bound $O(tn^{r-1}f_{d}(n))$. 
			\end{remark}
			
			\medskip \noindent \textbf{Case B: $G$ is a triangle.} Say $E(G) = \{(1,2),(2,3),(1,3)\}$. In this case, we first split $A_1$ and $A_2$ by Lemma \ref{lem:biclique}, then we split $A_2$ and $A_3$ and then we split $A_1$ and $A_3$.
			
			Using the induction hypothesis, by splitting $A_1,A_2$ we obtain
			\begin{equation}\label{eq:5}
				T_G(n) \leq O(b \log^{d-1} b)  \left( T_G(\tfrac{n}{b},  \tfrac{n}{b}, n, \ldots, n)+  O(tn^{r-1} f_{d}(n)) \right).
			\end{equation}
			By splitting $A_2,A_3$, we get
			\begin{equation}\label{eq:6}
				T_G(\tfrac{n}{b},  \tfrac{n}{b}, n, \ldots, n) \leq O(b \log^{d-1} b)  \left( T_G(\tfrac{n}{b},  \tfrac{n}{b^2},\tfrac{n}{b}, n, \ldots, n)+  O(tn^{r-1} f_{d}(n)) \right),
			\end{equation}
			and by splitting $A_1,A_3$ we have
			\begin{equation}\label{eq:7}
				T_G(\tfrac{n}{b},  \tfrac{n}{b^2}, \tfrac{n}{b}, n, \ldots, n) \leq O(b \log^{d-1} b)  \left( 
				T_G(\frac{n}{b^2},  \frac{n}{b^2},\frac{n}{b^2}, n, \ldots, n)+  O(tn^{r-1} f_{d}(n)) \right).
			\end{equation}
			Combining (\ref{eq:5})-(\ref{eq:7}), we obtain 
			\begin{equation}\label{eq:8}
				T_G(n) \leq O(b^3 \log^{3d-3} b)   \left(T_G(\tfrac{n}{b^2},\tfrac{n}{b^2},\tfrac{n}{b^2},n,\ldots,n)+ O (tn^{r-1} f_{d}(n))  \right),
			\end{equation}
			and by partitioning each of $A_4,\ldots,A_r$ arbitrarily into $b$ equal parts, we get
			\begin{equation}\label{eq:9}
				T_G(n) \leq O(b^{2r-3} \log^{3d-3} b) T_G(\tfrac{n}{b^2}) +  O(b^3 \log^{3d-3} b)O (tn^{r-1} f_{d}(n)).   
			\end{equation}
			As above, this recursion solves to $	T_G(n) =O(tn^{r-1} f_{d}(n)) $.
			
			\medskip \noindent \textbf{Case C: $G$ is a star.} Say $G$ is a star centered at $A_r$. In this case we don't apply the induction hypothesis. Instead, we apply Theorem~\ref{cons:lopsided} (only once, in contrast to the argument in Section \ref{sec:boxes_weak}).
			
			Assume on the contrary that  $T_G(n) >Ctn^{r-1} f_{d}(n) $ for some large constant $C$. Consider the intersection graph between the multiset $A_1 A_2 \ldots A_{r-1}= \{
			a_1 \cap \ldots\cap a_{r-1} : \forall i, a_i \in A_i, \cap_{i=1}^{r-1} a_i \neq \emptyset \}$ and the family $A_r$. Since the number of edges in this graph is at least $\Omega(tn^{r-1} f_{d}(n))$, by Theorem~\ref{cons:lopsided} (with the parameters $n, m=n^{r-1}$ and $g=n^{r-2}$), this intersection graph contains $K_{gt,t}$.
			
			This means that some $t$ specific boxes from $A_r$ intersect $gt=n^{r-2}t$ boxes of the type $a_1 \cap \ldots \cap a_{r-1}$ ($a_i \in A_i$). These $gt$ boxes involve at least $t$ boxes from each $A_i$ ($1 \leq i \leq r-1$). Hence, there exist $A'_1 \subset A_1, \ldots, A'_r \subset A_r$, each of size $t$, such that any $a_1,\ldots,a_r$, where $a_1 \in A'_1, \ldots, a_r \in A'_r$, are pairwise intersecting. (Here we use the setting of case~C in which for all $1 \leq i<j \leq r-1$, any two boxes $a_i \in A_i$ and $a_j \in A_j$ intersect.) Since axis-parallel boxes have Helly number 2, this implies that for all  $a_1 \in A'_1, \ldots, a_r \in A'_r$ we have $a_1 \cap \ldots \cap a_r \neq \emptyset$, and hence, the restriction of $H$ to $A'_1 \cup \ldots \cup A'_r$ is a copy of $K^r_{t,t,\ldots,t}$, a contradiction.  This completes the proof of Theorem~\ref{thm:unified}, and thus of Theorem~\ref{thm:boxes}. 			
		\end{proof}

		\section{Intersection Hypergraphs of Pseudo-Discs}
		\label{sec:pd}
		
		
			
In this section we prove Theorem \ref{thm:pd}. Let us recall the statement of the theorem. 
		
\medskip \noindent \textbf{Theorem \ref{thm:pd}.} Let $r,t \geq 2$ and let $H$ be the $r$-partite intersection hypergraph of families $A_1, \ldots, A_r$ of $n$ $y$-monotone pseudo-discs in general position in the plane. If $H$ is $K_{t,\ldots,t}^r$-free, then $|\E(H)|=O(t n^{r-1} (\log n)^{r-2} )$.
		
\medskip
		
A tool crucially used in the proof is a lopsided version of the following result from~\cite{CH23}.
\begin{theorem}\cite[Corollary~5.1]{CH23}
\label{thm:pdCH}
	Let $P$ be a set of $n$ points in the plane, and let $\F$ be a family of $m$ $y$-monotone pseudo-discs in general position in the plane. If the bipartite intersection graph $G(P,\F)$ is $K_{t,t}$-free ($t \geq 2$), then $$|E(G(P,\F))| = O(tn +tm \log \log m + \log t).$$
\end{theorem}
		
The lopsided variant of Theorem \ref{thm:pdCH} reads as follows. 

\begin{theorem}[Lopsided variant of Theorem \ref{thm:pdCH}]
\label{thm:pdls}
	Let $P$ be a multiset of $n$ points in the plane, and let $\F$ be a multiset of $m$ $y$-monotone pseudo-discs in general position in the plane. If the bipartite intersection graph $G(P,\F)$ is $K_{gt,t}$-free ($t,g \geq 2$, $gt$ on the side of the points), then $$|E(G(P,\F))| = O(tn +gtm \log m ).$$
\end{theorem}

Note that for $g=1$, the bound we obtain is weaker than the bound of Theorem~\ref{thm:pdCH}. The reason for this is explained in Remark \ref{rem:clip} below.

\subsection{Proof of Theorem \ref{thm:pdls}} \label{app:pd}

In the proof, we use the following variant of Matou{\v{s}}ek's shallow cuttings~(\cite{Mat92}, see also~\cite{CCCH12}).		
\begin{lemma}\cite[Theorem 5.1]{CH23} \label{lem:m/r}
	Let $F$ be a family of $y$-monotone pseudo-discs, $|F|=m$. Let $r,k \in N$. Then there exists an $\frac{1}{r}$-cutting of $F$, namely, a decomposition $\Xi$ of $\Re^d$ into $O(r^d)$ cells of constant descriptive complexity, such that the total weight of boundaries of shapes of $F$ intersecting a single cell is at most $m/r$.
\end{lemma}		
The following lemma is a lopsided version of \cite[Lemma 5.2]{CH23}:
\begin{definition}
	Given a family $F$ of sets, the \emph{depth} of a point (with respect to $F$) is the number of elements of $F$ that contain it.
\end{definition}
\begin{lemma}\label{lem:depth}
	Let $P$ be a set of $n$ points in $\Re^2$, and let $F$ be a family of $m$ $y$-monotone pseudo-discs. Let $g,t \in N$. If the bipartite intersection graph $G(P,F)$ is $K_{gt,t}$-free, then for any $r \leq \frac{m}{2t}$, the number of points of $P$ having depth between $\frac{m}{r}$ and $\frac{2m}{r}$ is at most $O(gtr)$. 
\end{lemma}

\begin{proof}
	Let $\Xi$ be an $\frac{1}{2r}$-cutting whose existence is guaranteed by Lemma \ref{lem:m/r}. By the last part of Lemma \ref{lem:m/r}, the number of cells in $\Xi$ that contain at least one point of depth between $\frac{m}{r}$ to $\frac{2m}{r}$ is $O(r)$. Each such cell $\nabla$ contains a point with depth$\geq \frac{m}{r}$ and intersects the boundaries of at most $\frac{m}{2r}$ pseudo-discs (since $\Xi$ is an $\frac{1}{2r}$-cutting). 
	
	Therefore, $\nabla$ is fully contained in at least $\frac{m}{2r} \geq t$ pseudo-discs. Since $G(P,F)$ is $K_{gt,t}$-free, $|\nabla \cap P| <gt$. Hence, the number of cells with a point of depth between $\frac{m}{r}$ and $\frac{2m}{r}$ is $O(r)$, and each such a cell contains at most $gt$ points, and so we are done.
\end{proof}		
Now we are ready to prove Theorem \ref{thm:pdls}. 



\begin{proof}[Proof of Theorem \ref{thm:pdls}]
	Let $2t=t_0<t_1< \ldots <t_{\ell}$ be parameters to be determined later, where $t_{\ell} \geq m$. We break $P$ into $O(\ell)$ classes as follows: $P_0 \subset P$ consists of the points whose depth (w.r.t. $F$) is below $t_0$, and for $i>0$, $P_i$ consists of the points whose depth is at least $t_{i-1}$ and smaller than $t_i$.
	
	By Lemma \ref{lem:depth} and by partitioning $P_i$ into the points with depth between $t_{i-1}$ to $2t_{i-1}$, between $2t_{i-1}$ to $4t_{i-1}$, etc., we can bound 
	$$|P_i| = O \left( gt  \left(  \frac{m}{t_{i-1}} + \frac{m}{2t_{i-1}}    + \frac{m}{4t_{i-1}} +\ldots  \right)  \right)  = O \left(  gt \frac{m}{t_{i-1}} \right).$$
	By Lemma~\ref{lem:m/r}, we compute a $\frac{t_i}{m}$-cutting and let $\Xi$ be the cells of the cutting that intersect $P_i$. By the last part of Lemma \ref{lem:m/r}, there are $O(\frac{m}{t_i})$ such cells. 
	
	Consider a cell $\nabla \in \Xi$. By the definition of $\Xi$, it contains a point of depth $\leq t_i$, and since $\Xi$ comes from a $\frac{t_i}{m}$-cutting, it intersects the boundaries of at most $t_i$ pseudo-discs. Hence, the total number of pseudo-discs intersecting or containing $\nabla$ is $O(t_i)$. 
	
	We can subdivide the simplices of $\Xi$ into $O(\frac{m}{t_i})$  subcells, such that each such subcell contains at most 
	$O \left(    \frac{ gt \frac{m}{t_{i-1}}  }{  \frac{m}{t_i} }   \right)  =  O \left( \frac{gt t_i}{t_{i-1}}     \right)$ points of $P_i$. 
	
	Define $I(x,y)$ to be the maximal possible value of $|E(G(P',F'))|$ for a set $P'$ of $x$ points and a family $F'$ of $y$ pseudo-discs. Using the above subdivision process, we obtain
	$$  I(|P_i|,|F|) \leq O\left(\frac{m}{t_i}\right) I  \left(   \frac{gtt_i}{t_{i-1}} , t_i \right).  
	$$
	For $i=0$, we use the trivial upper bound $  I(|P_0|,|F|) \leq O(|P_0| \cdot t_0) \leq O(nt)$ (the last inequality holds since $|P_0| \leq |P|=n$ and $t_0=2t$). The recursion obtained in this way is 
	$$  I(n,m) \leq \sum_{i=1}^{\ell}  O \left( \frac{m}{t_i} \right)   I  \left(  \frac{gtt_i}{t_{i-1}}  ,t_i \right)  +O(tn) .$$
	We choose the sequence $t_0=2t$ and $t_i=2t_{i-1}$ for $i>0$, and get 
	$$  I(n,m)  \leq \sum_{i=1}^{\log m}  O \left( \frac{m}{t_i} \right) I(2gt,t_i) +O(tn) 
	\leq \sum_{i=1}^{\log m}  O \left( \frac{m}{t_i} \right) O(gtt_i) +O(tn),$$
	where the right inequality holds since $I(2gt,t_i) \leq 2gtt_i$. Hence, $$I(n,m) = O(gtm \log m +tn),$$
	as asserted.
\end{proof}

\begin{remark} \label{rem:clip}
	For $g=1$, the bound we obtain in Theorem \ref{thm:pdls} is weaker than the bound in Theorem \ref{thm:pdCH}. The improved bound in \cite{CH23} is obtained by choosing carefully the values of the sequence $\{t_i\}_{i=0,1,\ldots}$ to make this sequence very-fast increasing in its last $\log \log m$ elements. It appears that in the lopsided setting of Theorem \ref{thm:pdls}, we cannot use a similar choice of the sequence, since this will make the dependency on $g$ polynomial instead of linear.
\end{remark}
		
\subsection{Proof of Theorem~\ref{thm:pd}}
		
\medskip The proof of Theorem \ref{thm:pd} uses induction on $r$, where the induction basis ($r=2$) is the following theorem from \cite{HMTS25}:
\begin{theorem}\cite[Theorem 1.6]{HMTS25} \label{thm:pdr=2}
	Let $t \geq 2$ and let $G$ be the bipartite intersection graph of two $y$-monotone families of pseudo-discs, each of size $n$. If $G$ is $K_{t,t}$-free then $|E(G)|=O(tn)$.
\end{theorem}

\begin{remark}
	Note that the lopsided Theorem~\ref{thm:pdls} builds upon Theorem~\ref{thm:pdCH}, rather than on  Theorem~\ref{thm:pdr=2} in which the dependence of the bound on $n$ is optimal. 
	If one proves a lopsided version of Theorem~\ref{thm:pdr=2}, this would yield the optimal bound $O_r(tn^{r-1})$ in Theorem~\ref{thm:pd}, by replacing Theorem~\ref{thm:pdls} with this lopsided version throughout the proof below. 
\end{remark}

\begin{proof}[Proof of Theorem \ref{thm:pd}]
	The proof is by induction on $r$, the base case being Theorem~\ref{thm:pdr=2} above. In the induction step, we use the following observation: If $r$ simple Jordan regions in $\Re^2$ intersect, then either
			\begin{itemize}
				\item There is a pair of regions $a_1,a_2$ such that an intersection point of their boundaries is contained in all other $r-2$ regions, or
				\item There is a region $a_1$ which is fully contained in all other $r-1$ regions. 
			\end{itemize}
	We call an intersection of the first type a \emph{type A} - intersection, and an intersection of the second type a \emph{type B} - intersection.
	
	For each intersecting pair $\{a_1,a_2\}$ of pseudo-discs we define a special point $p(a_1,a_2)$ as follows: If the intersection of $a_1$ and $a_2$ is of type A, then $p(a_1,a_2)$ is the left intersection point of the boundaries of $a_1$ and $a_2$. If this intersection is of type B, where $a_1 \subset a_2$, then $p(a_1,a_2)$ is the leftmost point of $a_1$.
	
	Without loss of generality, both for Type A and for Type B, we count intersections of the type $a_1 \cap \ldots \cap a_r \neq \emptyset $ ($a_i \in A_i$), where  $p(a_1,a_2)$ is contained in $a_3 \cap \ldots \cap a_r$. This affects the final bound by a multiplicative factor of $O_r(1)$.
	
	Let $P$ be the multiset $P=\{ p(a_1,\ldots,a_{r-1}): a_i \in A_i \}$, where $p(a_1,\ldots,a_{r-1})$ is $p(a_1,a_2)$ with multiplicity which is determined by the number of other tuples of $a_i$'s whose intersection contains it. Clearly, $|P| \leq n^{r-1}$. Let $G$ be the bipartite intersection graph of the multiset of points $P$ and the family of pseudo-discs $A_r$. If $|\E(H)|> C_r t n^{r-1} (\log n)^{r-2}$ (for a sufficiently large constant $C_r$), then by Theorem \ref{thm:pdls}, $G$ contains $K_{gt,t}$ for $g=C_{r-1} n^{r-2} (\log n)^{r-3}$ (where $C_{r-1}$ will be specified below, and $C_r$ is taken to be sufficiently large for the statement to hold by Theorem~\ref{thm:pdls}). This $K_{gt,t}$ comes from a set $S$ of $gt=C_{r-1}t n^{r-2} (\log n)^{r-3}$ distinct tuples of the form $(a_1,\ldots,a_{r-1})$ where $a_i \in A_i$, and 			
	a set $T$ of $t$ pseudo-discs from $A_r$. Let $Z=\bigcap\{a_r : a_r \in T\}$ be the intersection of all the pseudo-discs in $T$. Since all the elements of $T$ contain a common point (e.g., each point in $P$), $Z$ is non-empty, and it follows from properties of pseudo-discs families that $Z$ is connected. (For a proof of this geometric fact, see \cite[Theorem 4.4]{AKP20}).
	
	Now, we clip each pseudo-disc $s \in A_1 \cup \ldots \cup A_{r-1}$ to $s'=s \cap Z$, and slightly perturb the boundaries such that the family $\{s': s \in  A_1 \cup \ldots \cup A_{r-1}\}$ is still a family of pseudo-discs. This perturbation can be performed as follows: Define a partial ordering on $A_1 \cup \ldots \cup A_{r-1}$, by considering the intersection of each pseudo-disc with the boundary of $Z$ and ordering by inclusion. Extend the ordering into a linear ordering arbitrarily. Clip the pseudo-discs close to the boundary of $Z$ from the inside, in
	such a way that a ``smaller'' pseudo-disc (according to the linear ordering) is clipped closer to the boundary. In this way, an intersection is added only to pairs of 
	pseudo-discs whose boundaries have only one intersection point inside $Z$, and therefore, any two boundaries of $\{s': s \in  A_1 \cup \ldots \cup A_{r-1}\}$ intersect at most twice.
	
	Consider the $(r-1)$-partite intersection hypergraph $H'$ of the families $A_1', \ldots,A'_{r-1}$, where $A_i'=\{s':s\in A_i \}$. $H'$ has at least $|S|$ hyperedges, since for each $(a_1,\ldots,a_{r-1}) \in S$, the point $p(a_1,\ldots,a_{r-1})$ is contained in $Z$, and hence, $a'_1 \cap a'_2 \cap \ldots \cap a'_{r-1} \neq \emptyset$. Since $|S|> C_{r-1} t n^{r-2} (\log n)^{r-3}$, by the induction hypothesis $H'$ contains a $K_{t,\ldots,t}^{r-1}$ all of whose elements are fully contained in $Z=\bigcap\{a:a \in T\}$. (Here, $C_3$ is chosen to be sufficiently large for applying the inductive hypothesis.) Together with the set $T \subset A_r$, we obtain $K_{t,\ldots,t}^{r}$ in $H$. This completes the proof.
	\end{proof}
	
\begin{remark} 
	The clipping process is needed since without it, the $K_{t, \ldots,t}^{r-1}$ constructed in the proof may stem from intersections $a_1 \cap \ldots \cap a_{r-1}$ ($a_i \in A_i$) that are not related to intersection with the pseudo-discs in $T$. The clipping forces all those pseudo-discs to be included in $Z$, and in particular, forces their intersection points to be contained in $Z$, thus implying that the vertices of the $K_{t, \ldots,t}^{r-1}$ together with $T$ form a $K_{t, \ldots,t}^{r}$. We showed that such a clipping can be performed without losing the `pseudo-disc family' property.
\end{remark}

\section*{Acknowledgements}

We are grateful to Aleksa Milojevi{\'c} for useful discussions and for communicating to us the results of~\cite{HMTS25}, and to Parinya Chalermsook for kindly sharing with us the results of~\cite{COZ25}.

		\bibliographystyle{plain}
		\bibliography{references-klan}
		
		\appendix 
		
		\section{Lopsided Results for the Graph Zarankiewicz's Problem}\label{app:rect}
		
		In this appendix we prove Propositions~\ref{Prop:PtsBoxes} and~\ref{Prop:HorVer}, which give bounds for the lopsided Zarankiewicz problem for the intersection graphs of points and axis-parallel boxes in $\Re^d$ (Proposition~\ref{Prop:PtsBoxes}) and of horizontal and vertical segments in the plane (Proposition~\ref{Prop:HorVer}). 
		In the following, $G_{A,B}$ denotes the bipartite intersection graph of the families $A,B$.
		


		We start with the proof of Proposition~\ref{Prop:HorVer}. Let us recall its statement.
		
		\medskip \noindent \textbf{Proposition~\ref{Prop:HorVer}.}
		Let $A$ be a multiset of horizontal segments in $\mathbb{R}^2$ and let $B$ be a multiset of vertical segments in $\mathbb{R}^2$, where $|A|=n$ and $|B|=m$. Let $g>0$. If the bipartite intersection graph $G_{A,B}$ is $K_{t,gt}$-free (where the $t$ is on the side of $A$) then it has at most $27(t-1)m+27(gt-1)n$ edges.
		
		\medskip \noindent The proof-sketch presented below is an adaptation of the proof of Chalermsook, Orgo and Zarsav~\cite[Theorem~3]{COZ25} for the symmetric case (i.e., $g=1$). 
		As the proof in~\cite{COZ25} contains many technical details and the changes required for generalizing it to a general $g>0$ are rather straightforward, we describe the general structure of the proof, and then list the changes required for adapting the argument to all $g>0$.
		 
\begin{proof}[Proof-sketch]
		We show that there always exists $a \in A$ of degree at most $27(gt-1)$ or $b \in B$ of degree at most $27(t-1)$. Since this holds also for every induced subgraph of $G_{A,B}$, the required bound follows inductively.
		
		A segment $b \in B$ is called \emph{down heavy} (resp., \emph{up heavy}) with respect to $a \in A$ if it is intersected by at least $3(t-1)$ segments from $A$ whose $y$-coordinate is smaller (resp., larger) than the $y$-coordinate of $a$. (The segment $b$ may or may not intersect $a$). The proof is based on a sophisticated credit payment scheme, where $a \in A$ pays credits to part of the members of $B$. The credit payment is performed by the following two algorithms, in the first $a$ contributes credits to part of its neighbors in $G_{A,B}$ and in the second $a$ contributes credits to part of the members of $B$ that are not its neighbors in $G_{A,B}$.
		
		\begin{algorithm}[H]
			\caption{Neighbor Payment}
			\begin{algorithmic}[1]
				\For{each $a \in A$}
				\State Pay $\frac{9}{2}$ credits to each of the following neighbors of $a$ in $G_{A,B}$:
				\State \quad $(gt-1)$ leftmost up heavy elements of $B$ with respect to $a$.
				\State \quad $(gt-1)$ rightmost up heavy elements of $B$ with respect to $a$.
				\State \quad $(gt-1)$ leftmost down heavy elements of $B$ with respect to $a$.
				\State \quad $(gt-1)$ rightmost down heavy elements of $B$ with respect to $a$.
				\EndFor
			\end{algorithmic}
		\end{algorithm}
		
		\begin{algorithm}[H]
			\caption{Further Payment}
			\begin{algorithmic}[1]
				\For{each $a \in A$}
				\State Pay $\frac{9}{4}$ credits to each of the following elements of $B$ that do not intersect $a$ but intersect its rightward/leftward extension:
				\State \quad $(gt-1)$ rightmost up heavy elements of $B$ with respect to $a$ that intersect the leftward extension of $a$.
				\State \quad $(gt-1)$ leftmost up heavy elements of $B$ with respect to $a$ that intersect the rightward extension of $a$.
				\State \quad $(gt-1)$ rightmost down heavy elements of $B$ with respect to $a$ that intersect the leftward extension of $a$.
				\State \quad $(gt-1)$ leftmost down heavy elements of $B$ with respect to $a$ that intersect the rightward extension of $a$.
				\EndFor
			\end{algorithmic}
		\end{algorithm}
		
		The central part of the proof is to show (Lemma~\ref{lem:Cha} below) that for every $b \in B$ of degree at least $27(t-1)$ in $G_{A,B}$, the total sum of credits that $b$ receives, $\text{crd}(b)$, is larger than the degree of $b$ in $G_{A,B}$. Given this lemma, it is easy to show that there exists $a \in A$ of degree at most $27(gt-1)$ or $b \in B$ of degree at most $27(t-1)$.
		Indeed, assume that for every $b \in B$ we have $\deg_{G_{A,B}}(b) > 27(t-1)$. Since every $a \in A$ pays at most 
		$$\frac{9}{2} \cdot 4 \cdot (gt-1) + \frac{9}{4} \cdot 4 \cdot (gt-1) = 27(gt-1)$$
		credits, it follows from Lemma~\ref{lem:Cha} that
		$$|E(G_{A,B})| = \sum_{b \in B} \deg_{G_{A,B}}(b) \leq \sum_{b \in B} \text{crd}(b) \leq 27(gt-1)n.$$
		By averaging over the members of $A$ we get that there exists $a \in A$ of degree at most $27(gt-1)$.
		
		Therefore, all that remains for us to prove is Lemma~\ref{lem:Cha}.
		
		\begin{lemma}\label{lem:Cha}
			For every $b \in B$ of degree at least $27(t-1)$ in $G_{A,B}$, we have $\deg_{G_{A,B}}(b) < \text{crd}(b)$.
		\end{lemma}
		
		We now describe the steps of the proof of Lemma \ref{lem:Cha} with references to the corresponding claims in~\cite{COZ25}. Let $b \in B$ be such that $\deg(b) \geq 27(t-1)$. The proof is based on dividing the horizontal segments intersecting $b$ into $\ell$ sequences of $3(t-1)$ consecutive segments (possibly leaving out the last few segments, if $3(t-1) \nmid \deg(b)$), and proving a claim (corresponding to Lemma~23 in~\cite{COZ25}) which says that $b$ receives at least $\frac{9}{2}(t-1)$ credits from members of $A$ whose $y$-coordinate is within each of the $\ell$ sequences, except the top and bottom ones. A simple calculation (Lemma~24 in~\cite{COZ25}, whose statement and proof work here verbatim without any change in the parameters) proves that the assertion $\deg_{G_{A,B}}(b) < \text{crd}(b)$ follows.
		
		The proof of the claim (Lemma~23 in~\cite{COZ25}), presented in~\cite[Appendix~D]{COZ25}, is sophisticated and technical. For the reader's convenience, we indicate the (completely technical) changes required for generalizing this argument to all $g>0$. 
		
		\begin{enumerate}
			\item In the statement of Lemma~35 in~\cite{COZ25}, no change is required. In the proof of Lemma~35, the only change is that the size of $Q$ (in the notation of~\cite{COZ25}) is $gt-1$ instead of $t-1$, since the forbidden biclique is $K_{t,gt}$.
			\item Lemma~36 in~\cite{COZ25} is symmetric to the aforementioned Lemma~35.
			\item In the statement of Lemma~23 in~\cite{COZ25}, no change is required. In the proof of the lemma, again, the only change is that the set $Q$ has size $gt-1$ instead of $t-1$, for the same reason.
		\end{enumerate}
\end{proof}

		The following lemma is a lopsided version of the bound of~\cite{CH23} on the number of edges in $K_{t,t}$-free bipartite intersection graphs of points and bottomless rectangles. For the sake of simplicity, from now on we do not track the exact values of the constants and use the $O(\cdot)$ notation instead.
		
		\begin{lemma}\label{Prop:Bottomless}
			Let $A$ be a multiset of points in $\Re^2$ and let $B$ be a multiset of bottomless rectangles, where $|A|=n$ and $|B|=m$. Let $g>0$. 
			\begin{enumerate}
				\item If $G_{A,B}$ is $K_{t,gt}$-free ($t$ on the side of $A$) then it has $O(gt n+ tm)$ edges.
				
				\item If $G_{A,B}$ is $K_{gt,t}$-free ($t$ on the side of $B$) then it has $O(gt m+ tn)$ edges. 
			\end{enumerate}
		\end{lemma}
		
		\begin{proof}
			We represent each bottomless rectangle by its `upper' edge (thus obtaining a family $B'$), and represent each point by a long segment emanating from it in the `upper' direction (thus obtaining a family $A'$). Clearly, if the segments are sufficiently long, the bipartite intersection graph $G_{A',B'}$ is isomorphic to $G_{A,B}$. Note that $A'$ and $B'$ are multisets of horizontal and vertical segments in the plane, respectively. Hence, (2) follows by applying Prop~\ref{Prop:HorVer} to $A',B'$, and (1) follows by rotating the plane by $\pi/2$ (so that the roles of vertical and horizontal segments are interchanged) and then applying Proposition~\ref{Prop:HorVer} to $B',A'$.  
		\end{proof}
		
		The following proposition is a lopsided version of the bound of~\cite[Lemma~4.4]{CH23} on the number of edges in $K_{t,t}$-free bipartite intersection graphs of points and rectangles.
		
		\begin{proposition}\label{Prop:PtsRectangles}
			Let $A$ be a multiset of points in $\Re^2$ and let $B$ be a family of rectangles, where $|A|=n$ and $|B|=m$. Let $g,\epsilon>0$. 
			\begin{enumerate}
				\item If $G_{A,B}$ is $K_{t,gt}$-free ($t$ on the side of $A$) then it has $O_{\epsilon} \left(gt n \frac{\log n}{\log \log n}+ tm \log^{\epsilon}n\right)$ edges.
				
				\item If $G_{A,B}$ is $K_{gt,t}$-free ($t$ on the side of $B$) then it has $O_{\epsilon} \left(t n \frac{\log n}{\log \log n}+ gt m \log^{\epsilon}n\right)$ edges.
			\end{enumerate}
		\end{proposition}
		
		The proof uses a divide-and-conquer argument, with a parameter $b$. It is clearly sufficient to prove the following claim, where $b$ is a parameter that may depend on $n,m$.
		\begin{claim}\label{Claim:PtsRectangles}
			Let $A$ be a multiset of points in $\Re^2$ and let $B$ be a multiset of axis-parallel rectangles, where $|A|=n$ and $|B|=m$. Let $b$ be a parameter.
			\begin{enumerate}
				\item If $G_{A,B}$ is $K_{t,gt}$-free ($t$ on the side of $A$) then it has $O(gt n \log_b n + btm)$ edges.
				
				\item If $G_{A,B}$ is $K_{gt,t}$-free ($t$ on the side of $B$) then it has $O \left(t n \log_b n+ bgt m\right)$ edges.
			\end{enumerate}
		\end{claim}
		\medskip Indeed, substituting $b=(\log n)^{\epsilon}$ into the claim yields the assertion of Proposition~\ref{Prop:PtsRectangles}.
		
		\begin{proof}[Proof of Claim~\ref{Claim:PtsRectangles}]
			Let $A,B$ be multisets that satisfy the assumptions of the claim. We divide the plane into $b$ vertical strips $\sigma_1,\sigma_2,\ldots,\sigma_b$, such that each strip contains $\frac{n}{b}$ points of $A$. For each $1 \leq i \leq b$, we denote by $A_i$ the points of $A$ in $\sigma_i$, by $B_i$ the rectangles of $B$ which are fully included in $\sigma_i$, and by $B'_i$ the rectangles of $B$ which intersect $\sigma_i$ but are not included in it. Denote $|B_i|=m_i$, and $|\cup_i B_i'|=m_0$ (so $\sum_{i=0}^b m_i=m$). Clearly, we have 
			\[
			|E(G_{A,B})| = \sum_{i=1}^b I(A_i,B_i) + \sum_{i=1}^b I(A_i,B_i'),
			\]
			where for any $X,Y$, $I(X,Y):=|E(G_{X,Y})|$. 
			
			\medskip \noindent \emph{Proof of~(1).} Denote by $f(n,m)$ the maximum number of edges in such a bipartite intersection graph of a multiset of $n$ points and a multiset of $m$ rectangles. 
			
			For the terms $I(A_i,B_i)$, we clearly have
			\[
			\sum_{i=1}^b I(A_i,B_i) \leq \sum_{i=1}^b f \left(\tfrac{n}{b},m_i \right).
			\]
			To handle the terms $I(A_i,B'_i)$, we note that for each $i$ and each $y \in B'_i$, the intersection $y \cap \sigma_i$ is a `truncated rectangle' in which either the right side, the left side, or both are missing. We may assume w.l.o.g. that there are no $y$'s with both sides missing, by `closing up' one side close to the end of the strip, in a way that does not change $I(A_i,B'_i)$. We divide $B'_i$ into the multiset $B'_{i,L}$ of rectangles missing the left side and the multiset $B'_{i,r}$ of rectangles missing the right side. 
			
			The bipartite intersection graph $G_{A_i,B'_{i,L}}$ can be viewed as the bipartite intersection graph of a multiset of points and a multiset of bottomless rectangles (where the plane is rotated counterclockwise by $3\pi/2$). Hence, by Lemma~\ref{Prop:Bottomless}(1), we have
			\[
			I(A_i,B'_{i,L}) \leq O(gt |A_i| + t |B'_{i,L}|).
			\]
			The same argument applies for the bipartite intersection graph $G_{A_i,B'_{i,R}}$, and hence, we have
			\[
			I(A_i,B'_i) = I(A_i,B'_{i,L}) + I(A_i,B'_{i,R}) \leq O(gt |A_i| + t |B'_{i}|).
			\]
			Therefore,
			\[
			\sum_{i=1}^b I(A_i,B'_i) \leq \sum_{i=1}^b O(gt |A_i| + t |B'_{i}|) \leq O(gt n + tb m_0),
			\]
			where the last inequality holds since each $y \in B$ intersects at most $b$ strips. 
			
			Combining with the bound on the terms $I(A_i,B_i)$, we obtain
			\[
			I(A,B) = \sum_{i=1}^b I(A_i,B_i) + \sum_{i=1}^b I(A_i,B_i') \leq \sum_{i=1}^b f \left(\tfrac{n}{b},m_i \right) + O(gt n + tb m_0).
			\]
			As this holds for any multisets $A,B$ that satisfy the assumptions of the claim, we get the recursive formula
			\begin{equation}
				f(n,m) \leq \max_{\{m_0,m_1,\ldots,m_b:m_0+\ldots+m_b=m\}} \sum_{i=1}^b f \left(\tfrac{n}{b},m_i \right) + O(gt n + tb m_0),
			\end{equation}
			which solves to 
			\[
			f(n,m) = O(gt n \log_b n + tbm),
			\]
			as asserted.
			
			\medskip \noindent \emph{Proof of~(2).} Denote by $g(n,m)$ the maximum number of edges in a bipartite intersection graph of a multiset of $n$ points and a multiset of $m$ rectangles that satisfy the assumptions.  The proof is similar to the proof of~(1), with Lemma~\ref{Prop:Bottomless}(2) replacing Lemma~\ref{Prop:Bottomless}(1). Specifically, for each $1 \leq i \leq b$, we may use Lemma~\ref{Prop:Bottomless}(2) to obtain
			\[
			I(A_i,B'_i) \leq O(t|A_i|+gt |B'_i|), 
			\]
			and consequently,
			\[
			\sum_{i=1}^b I(A_i,B'_i) \leq O(tn+ gt bm_0).
			\]
			This yields the recursive formula
			\begin{equation}
				g(n,m) \leq \max_{\{m_0,m_1,\ldots,m_b:m_0+\ldots+m_b=m\}} \sum_{i=1}^b g \left(\tfrac{n}{b},m_i \right) + O(tn+ gt bm_0),
			\end{equation}
			which solves to
			\[
			g(n,m) = O(t n \log_b n + gt b m),
			\]
			as asserted.
		This completes the proof of Claim~\ref{Claim:PtsRectangles} and of Proposition~\ref{Prop:PtsRectangles}.
		\end{proof}

		Now we are ready to prove Proposition~\ref{Prop:PtsBoxes} which is a lopsided version of the bound of~\cite[Theorem~4.5]{CH23} on the number of edges in $K_{t,t}$-free bipartite intersection graphs of points and axis-parallel boxes in $\Re^d$. Let us recall the statement of the proposition. 

\medskip \noindent \textbf{Proposition~\ref{Prop:PtsBoxes}.}
	Let $A$ be a multiset of points in $\Re^d$ and let $B$ be a multiset of axis-parallel boxes, where $|A|=n$ and $|B|=m$. Then for any $\epsilon>0$,
	\begin{enumerate}
	 	\item If the bipartite intersection graph $G_{A,B}$ is $K_{t,gt}$-free ($t$ on the side of $A$) then it has \newline $O_{\epsilon} \left(gt n (\frac{\log n}{\log \log n})^{d-1}+ tm (\frac{\log n}{\log \log n})^{d-2+\epsilon}\right)$ edges.
		
		\item If the bipartite intersection graph of $G_{A,B}$ is $K_{gt,t}$-free ($t$ on the side of $B$) then it has $O_{\epsilon} \left(t n (\frac{\log n}{\log \log n})^{d-1}+ gt m (\frac{\log n}{\log \log n})^{d-2+\epsilon}\right)$ edges.
	\end{enumerate}

Like in the proof of Proposition~\ref{Prop:PtsRectangles} above, it is clearly sufficient to prove the following claim, where $b$ is a parameter that may depend on $n,m$.
\begin{claim}\label{Claim:PtsBoxes}
	Let $A$ be a multiset of points in $\Re^d$ and let $B$ be a multiset of axis-parallel boxes, where $|A|=n$ and $|B|=m$. Let $b$ be a parameter.
	\begin{enumerate}
		\item If $G_{A,B}$ is $K_{t,gt}$-free ($t$ on the side of $A$) then it has $O(gt n (\log_b n)^{d-1} + tm b^{d-1} (\log_b n)^{d-2})$ edges.
		
		\item If $G_{A,B}$ is $K_{gt,t}$-free ($t$ on the side of $B$) then it has $O \left(t n (\log_b n)^{d-1}+ gt m b^{d-1} (\log_b n)^{d-2}\right)$ edges.
	\end{enumerate}
\end{claim}

\medskip Indeed, substituting $b=(\log n)^{\frac{\epsilon}{d-1}}$ into the claim yields the assertion of Proposition~\ref{Prop:PtsRectangles}.

\begin{proof}[Proof of Claim~\ref{Claim:PtsBoxes}]
	The proof is by induction on $d$. The induction basis is the case $d=2$ proved in Claim~\ref{Claim:PtsRectangles}. 
	In the induction step, we assume that the claim holds for dimension $d-1$ and prove it for dimension $d$. 
	
\medskip \noindent \emph{Proof of~(1).}	We use the following auxiliary notion. Let $I_d(n,v)$ be the maximal number of intersections between a multiset $A$ of points and a multiset $B$ of axis-parallel boxes inside a `vertical strip' $\mathcal{U}=\{x \in \Re^d:u_L<x_1<u_R\}$ of $\Re^d$, where:
	\begin{enumerate}
		\item Each box in $B$ has either half of its vertices or all of its vertices in $\mathcal{U}$;
		
		\item The total number of vertices of boxes in $B$ inside $\mathcal{U}$ is $m \cdot 2^{d-1}$;
		
		\item The bipartite intersection graph of $A,B$ is $K_{t,gt}$-free.
	\end{enumerate}
	We shall prove that 
	\[
	I_d(n,m) \leq O(gt n (\log_b n)^{d-1} +  tm b^{d-1} (\log_b n)^{d-2}).
	\] 
	This clearly implies the assertion, as by considering a vertical strip $\mathcal{U}$ that fully contains all points in $A$ and boxes in $B$ (where $|A|=n$ and $|B|=m$), we get $E(G_{A,B}) \leq I_d(n,2m)$. (Note that we neglect factors of $O_d(1)$). 
	
	Let $A,B$ be multisets that satisfy assumptions~(1)--(3) with respect to a strip $\mathcal{U} \subset \Re^d$. We divide $\mathcal{U}$ into $b$ vertical sub-strips $\sigma_1,\ldots,\sigma_d$ such that $\sigma_i=\{x \in \Re^d: u_i<x_1<u_{i+1}\}$ (where $u_0=u_L$ and $u_{d+1}=u_R$), in such a way that each sub-strip contains $\frac{n}{b}$ points in $A$. For $i=1,\ldots,b$, we denote by $A_i$ the points of $A$ in $\sigma_i$, by $B_i$ the boxes in $B$ that have at least one vertex in $\sigma_i$, and by $B'_i$ the boxes in $B$ that intersect $\sigma_i$ but do not have vertices in it. Note that each box of $B$ belongs to either one or two multisets $B_i$ and to at most $b$ multisets $B'_i$. We also denote the number of vertices of boxes in $B$ contained in $\sigma_i$ by $2^{d-1} \cdot m_i$.
	
	The number of intersections in $\mathcal{U}$ between points in $A$ and boxes in $B$ is clearly 
	\begin{equation}\label{Eq:Bound-boxes1}
		\sum_{i=1}^b I(A_i,B_i) + I(A_i,B'_i).
	\end{equation}
	By the definitions, for any $i$ we have $I(A_i,B_i)\leq I_d(\frac{n}{b},m_i)$ (where $\sigma_i$ is taken as the vertical strip instead of $\mathcal{U}$).  
	
	To handle the terms $I(A_i,B'_i)$, we observe that a point in $A_i$ intersects a box in $B'_i$ if and only if their projections on the hyperplane $\mathcal{H}=\{x \in \Re^d: x_1=u_i\}$ (which are a point and a $(d-1)$-dimenstional box) intersect. Let $\bar{A}_i$ and $\bar{B}'_i$ denote the corresponding multisets of projections. We have $|\bar{A}_i|=|A_i|=\frac{n}{b}$ and $|\bar{B}'_i|=|B'_i|\leq m$. The multisets $\bar{A}_i$ and $\bar{B}'_i$ are a multiset of points and a multiset of axis-parallel boxes in $\Re^{d-1}$ whose bipartite intersection graph is $K_{t,gt}$-free. Hence, by the induction hypothesis we have
	\begin{align*}
		\begin{split}
			I(A_i,B'_i)=I(\bar{A}_i,\bar{B}'_i) &\leq 
			O(gt \tfrac{n}{b} (\log_b \tfrac{n}{b})^{d-2} +  tm b^{d-2} (\log_b \tfrac{n}{b})^{d-3}) \\
			&\leq O(gt  \tfrac{n}{b} (\log_b n)^{d-2} +  tm b^{d-2} (\log_b n)^{d-3}).
		\end{split}
	\end{align*}
	Combining the bounds on $I(A_i,B_i)$ and $I(A_i,B'_i)$, we obtain the recursive formula 
	\[
	I_d(n,m) \leq \max_{m_1+\ldots+m_b=m} \sum_{i=1}^b \left(I_d(\tfrac{n}{b},m_i)+O(gt  \tfrac{n}{b} (\log_b n)^{d-2} +  tm b^{d-2} (\log_b n)^{d-3}) \right), 
	\]
	which solves to
	\begin{align*}
		\begin{split}
			I_d(n,m) &\leq \log_b n \cdot O(gt n (\log_b n)^{d-2} +  tm b^{d-2} (\log_b n)^{d-3}) \\ 
			&\leq O(gt n (\log_b n)^{d-1} +  tm b^{d-1} (\log_b n)^{d-2}),
		\end{split}
	\end{align*}
	as asserted.

	\medskip \noindent \emph{Proof of~(2).} The proof is similar to the proof of~(1), with Claim~\ref{Claim:PtsRectangles}(2) replacing Claim~\ref{Claim:PtsRectangles}(1). Specifically, ~\eqref{Eq:Bound-boxes1} holds without change and we have 
	$I(A_i,B_i)\leq I_d(\tfrac{n}{b},m_i)$, exactly as in the proof of~(1). By the induction hypothesis, we get 
\begin{align*}
	\begin{split}	
	I(A_i,B'_i)=I(\bar{A}_i,\bar{B}'_i) &\leq 
	O \left(t \tfrac{n}{b} (\log_b \tfrac{n}{b})^{d-2}+ gt m b^{d-2} (\log_b \tfrac{n}{b})^{d-3}\right) \\
	&\leq O \left(t \tfrac{n}{b} (\log_b n)^{d-2}+ gt m b^{d-2} (\log_b n)^{d-3}\right).  
	\end{split}
\end{align*}
Combining the bounds on $I(A_i,B_i)$ and $I(A_i,B'_i)$, we obtain the recursive formula 
\[
I_d(n,m) \leq \max_{m_1+\ldots+m_b=m} \sum_{i=1}^b \left(I_d(\tfrac{n}{b},m_i)+ O \left(t \tfrac{n}{b} (\log_b n)^{d-2}+ gt m b^{d-2} (\log_b n)^{d-3}\right) \right), 
\]
which solves to
\begin{align*}
\begin{split}
I_d(n,m) &\leq b \log_b n \cdot  O \left(t \tfrac{n}{b} (\log_b n)^{d-2}+ gt m b^{d-2} (\log_b n)^{d-3}\right) \\ 
&\leq O(t n (\log_b n)^{d-1} +  gt m b^{d-1} (\log_b n)^{d-2}),
\end{split}
\end{align*}
as asserted.
This completes the proof of Claim~\ref{Claim:PtsBoxes} and of Proposition~\ref{Prop:PtsBoxes}.
\end{proof}

\end{document}